\documentclass[]{amsart}

\usepackage{amssymb, mathrsfs, color, amsthm}

\newtheorem{thm}{Theorem}[section]
\newtheorem{cor}[thm]{Corollary}
\newtheorem{lem}[thm]{Lemma}

\newtheorem{prop}[thm]{Proposition}

\numberwithin{equation}{section}

\renewcommand{\phi}{\varphi}

\newcommand{\D}{\mathbb{D}}

\newcommand{\eps}{\varepsilon}

\newcommand{\B}{\mathcal{B}}

\newcommand{\re}{\mbox{\rm Re\,}}

\newcommand{\diam}{\mbox{\rm diam\,}}

\begin{document}

\title[Composition Semigroups on BMOA and $H^{\infty}$]
{Composition Semigroups on BMOA and $H^{\infty}$}
\date{\today}

\author[A. Anderson]{Austin Anderson}
\address{Department of Mathematics, University of Hawaii, Honolulu, Hawaii 96822}
\email{austina@hawaii.edu}
\author[M. Jovovic]{Mirjana Jovovic}
\address{Department of Mathematics, University of Hawaii, Honolulu, Hawaii 96822}
\email{jovovic@math.hawaii.edu}
\author[W. Smith]{Wayne Smith}
\address{Department of Mathematics, University of Hawaii, Honolulu, Hawaii 96822}
\email{wayne@math.hawaii.edu}

\thanks {}
\subjclass [2000] {Primary: 47B33; Secondary: 47D05, 30A76}
\keywords{Composition semigroups, BMOA, $H^\infty$.
 }

\begin{abstract}
 We study $[\phi_t , X]$,  the maximal space of strong continuity for a semigroup of
 composition operators induced by a semigroup $\{\phi_t\}_{t\ge0}$ of analytic self-maps of the unit disk,
  when $X$ is BMOA, $H^\infty$ or the disk algebra. In particular, we show that $[\phi_t,\text{BMOA}] \neq \text{BMOA}$ for all nontrivial semigroups.   We also prove, for every semigroup $\{\phi_t\}_{t\ge0}$,  that
$\lim_{t \to 0^+} \phi_t(z) = z$ not just pointwise, but in $H^{\infty}$ norm.
This provides a unified proof of known results about $[\phi_t , X]$ when $X \in \{H^p, A^p, \mathcal B_0, \text{VMOA}\}$.

\end{abstract}
\maketitle

\section{Introduction}

Let $\D$ denote the unit disk $\{z : |z| < 1\}$ and $H(\D)$ the set of analytic functions on $\D$.
A one-parameter semigroup $\{\phi_t\}_{t\ge0}$ of analytic functions on  $\D$ is a family  of
analytic functions $\phi_t : \D \to \D$ that satisfies the following three conditions:
\begin{enumerate}
\item [(SG1)] $\phi_0$ is the identity, i.e. $\phi_0(z)=z$, $z\in \D$;
\item [(SG2)] $\phi_{s+t}=\phi_s\circ\phi_t$, for all $t,s\ge 0$;
\item [(SG3)] the mapping $(t,z)\to\phi_t(z)$ is continuous on $[0,\infty)\times \D$.
\end{enumerate}

In the trivial case, $\phi_t(z)=z$ for all $t\ge0$.  Otherwise, we say that $\{\phi_t\}$ is nontrivial.
For $\{\phi_t\}$, consider the  set of linear operators $\{ C_t \}$, where $C_t(f) = f \circ \phi_t$ for $f \in H(\D)$.  Then $C_0$ is the identity operator and $C_{s+t}=C_s C_t$, and so
$\{ C_t \}$ is called a composition semigroup.
If $C_t$ is a bounded operator on some Banach space $X\subset H(\D)$ for all $t \geq 0$, we say that
the  semigroup $\{\phi_t\}$ acts on $X$.
If in addition
the strong continuity condition
  $$\lim_{t \to 0^+} \| f\circ \phi_t - f \|_X = 0$$
holds for all $f \in X$, then it is said that $\{\phi_t\}$ \textit{generates} $\{ C_t \}$, and $\{C_t\}$ is strongly continuous on $X$.

The study of composition semigroups on spaces of analytic functions began in \cite{BP}, where Berkson
and Porta investigated the basic structure of semigroups $\{\phi_t\}$ and properties of $\{C_t\}$ on
the Hardy spaces of the disk.  A survey of further developments through 1998 of
semigroups of composition operators
acting on classical Banach spaces can be found in \cite{Sis}.
More recent papers about composition semigroups  include  \cite{Bla}, \cite{Bla2} and  \cite{CD}.
Banach spaces considered include the Hardy spaces $H^p$, the Bergman spaces $A^p$,
the disk algebra $A$, the Dirichlet space $\mathcal D$, the Bloch space $\mathcal B$, BMOA, and
their ``little-oh" subspaces $\mathcal B_0$ and VMOA.

For a semigroup $\{\phi_t\}$ that acts on a Banach space $X \subset H(\D)$,
denote by $[\phi_t , X]$ the maximal closed subspace of $X$ on which $\{ C_t \}$ is strongly continuous.
Thus $[\phi_t , X]=X$ means that $\{\phi_t\}$  generates $\{ C_t \}$, and $\{C_t\}$ is strongly continuous on $X$.
Here is a quick summary of what is known:
\begin{enumerate}
\item [(i)] If $X\in\{H^p\,\,(1\le p<\infty), A^p\,\,(1\le p<\infty), \mathcal D, \mathcal B_0, \text{VMOA}\}$
and $\{\phi_t\}$ is any  semigroup, then $[\phi_t , X]=X$;
\item [(ii)] For every nontrivial  semigroup $\{\phi_t\}$,
$[\phi_t ,H^\infty]\subsetneqq H^\infty$ and
$[\phi_t ,\mathcal B]\subsetneqq \mathcal B$;
\item [(iii)] For every  semigroup $\{\phi_t\}$,
$\text{VMOA} \subseteqq [\phi_t ,\text{BMOA}]$ and
$\mathcal B_0\subseteqq[\phi_t ,\mathcal B]$.
\end{enumerate}

Statement (ii) can be proved using a
functional analytic argument based on
$H^{\infty}$ and the Bloch space having the Dunford-Pettis property; see \cite{Bla2} and \cite{Lotz}.
The space BMOA does not have the Dunford-Pettis property, so this proof is not available and the
corresponding statement had remained open.  Our first result resolves this problem by
extending statement (ii) to include BMOA.
The proof is function theoretic, and also provides a new proof of the result for $H^\infty$ and the Bloch space;
see Section \ref{bmoasec}.

\begin{thm} \label{bmoathm}
Suppose $H^\infty \subseteq X \subseteq \mathcal B$.  Then $[\phi_t,X] \subsetneqq X$.  In particular, $[\phi_t, \text{BMOA}] \subsetneqq \text{BMOA}$.
\end{thm}

From (SG1) and (SG3) we have the pointwise convergence $\phi_t(z) \to z$ as $t \to 0^+$, which can easily be
extended to uniform convergence on compact subsets of $\D$.
It was recently observed by P. Gumenyuk that  this extends to uniform convergence \textit{on
all of $\D$}
for every  semigroup $\{\phi_t\}$.

\begin{thm} \label{uniformthm}  \cite[Proposition 3.2]{Gum}
For every   semigroup $\{ \phi_t \}$,
$$
\lim_{t \to 0^+} \|\phi_t(z) - z\|_{H^\infty} = 0.
$$
\end{thm}

We thank the referee for giving us this reference.
Our proof, given in Section 4, and that in {\cite{Gum}}
are based on the fact that univalent functions have no Koebe arcs.

An easy consequence of this theorem is the following corollary:

\begin{cor} \label{polythm}
Let $X$ be a Banach space that contains $H^{\infty}$, and let $X_{\mathcal P}$ be the closure of the polynomials in $X$.  For all semigroups $\{\phi_t\}$,
$X_{\mathcal P} \subset [\phi_t,X]$.
\end{cor}

We note that Corollary \ref{polythm} provides a unified proof of some of the known results mentioned above.  In particular,
the corollary gives an alternate proof of the cases that $X \in \{H^p, A^p, \mathcal B_0, \text{VMOA}\}$
in statement (i)
above, since polynomials are dense in these spaces, and also includes all of statement (iii).
It also establishes that the natural extension of
statement (iii) to $H^\infty$ is valid:

\begin{cor} \label{hinftythm}
For every  semigroup $\{\phi_t\}$, the disk algebra $A$ satisfies
$$
A \subset[\phi_t ,H^\infty].
$$
 \end{cor}

It is well known that a composition operator $C_\phi$ is bounded on the disk algebra $A$ if and only if
$\phi\in  A$.  Hence a semigroup $\{\phi_t\}$ acts on $A$ if and only if $\{\phi_t\}\subset A$.  Another
consequence of Theorem \ref{uniformthm} is the following:
 \begin{cor} \label{Athm} A semigroup $\{\phi_t\}$ satisfies
$\{ \phi_t \} \subset A$ if and only if $[\phi_t,A] = A$.
 \end{cor}
Corollary \ref{Athm} should be compared to \cite[Theorem 1.1]{CD}, where additional conditions are included with the statement that $\{\phi_t\}\subset A$.
Theorem \ref{uniformthm} allows removal of these additional conditions.

The rest of the paper is organized as follows:  Section \ref{background} contains background material.
 Theorem \ref{bmoathm}
is proved in Section \ref{bmoasec},  via another theorem involving estimates of the Bloch norm.
Theorem \ref{uniformthm} and our results
about $[\phi_t ,H^\infty]$ are in Section \ref{Asec}.  In addition to the theorems above, we show that for
some  semigroups we have $A=[\phi_t ,H^\infty]$, while   $A \subsetneqq [\phi_t ,H^\infty]$
for others.

\vskip .2in

{\it Notation for constants.}
For $X$ and $Y$ nonnegative quantities, the notation $X\lesssim Y$  or $Y\gtrsim X$
means $X\le CY$, where the exact value of the constant $C$ is not important. Similarly,  $X\approx Y$ means that
both $X\lesssim Y$ and $Y\lesssim X$ hold.

\section{Background} \label{background}

We begin with some basic definitions.

$H^{\infty}$ is the space of bounded functions in $H(\D)$, with norm
  $$\|f\|_{\infty} = \sup \{|f(z)| : z \in \D\}.$$

Let $A$ denote the disk algebra, i.e., the subspace of $H^{\infty}$ consisting of functions that extend to be continuous on $\overline{\D}$.

$\text{BMOA}$  is the Banach space of all analytic functions in the Hardy space $H^2$ whose boundary values have bounded mean oscillation. There are many characterizations of this space; we will use the one in terms of Carleson measures; see \cite{Gar}, \cite{Zh}. A function $f \in H^2$ belongs to $\text{BMOA}$ if and only if there exists a constant $C > 0$ such that
$$\int_{R(I)} |f'(z) |^2 (1 - |z|^2)\;dA (z) \leq C|I|,$$
for every arc $I \subset \partial \D$, where $R(I)$ is the Carleson rectangle determined by $I$.  That is, $$R(I) = \{re^{i \theta} \in \D: 1 - \frac{|I|}{2 \pi} < r < 1,  e^{i \theta} \in I \},$$
where $|I|$ denotes the length of $I$ and $dA(z)$ is the normalized Lebesgue measure on $\D$. The corresponding $\text{BMOA}$ norm is
$$\|f\|_* = |f(0)| + \sup_{I \subset \partial \D} \left(\frac{1}{|I|} \int_{R(I)} |f'(z)|^2 (1 - |z|^2) \; dA(z) \right)^{1/2}.$$
Trivially, each polynomial belongs to BMOA. The closure of all polynomials in $\text{BMOA}$ is denoted by $\text{VMOA}$.  Alternatively,
$$\text{VMOA} = \{ f \in \text{BMOA}: \lim_{|I| \rightarrow 0} \frac{1}{|I|} \int_{R(I)} |f'(z)|^2 (1 - |z|^2) \; dA(z) = 0 \}.$$

The Bloch space is
$$\B = \{ f \in H(\D) :  \|  f  \|_{\B}  = |f(0)| + \sup_{z \in \D} |f'(z)|(1-|z|^2) < \infty \}.$$
The closure of all polynomials in $\B$ is denoted by $\B_0$, the little Bloch space. It is usually defined as
$$\B_0 = \{ f \in H(\D) :  \lim_{|z| \rightarrow 1} |f'(z)|(1-|z|^2) = 0 \}.$$

It is well known that  $H^{\infty} \subset \text{BMOA} \subset \B$, and $\|f\|_{\infty} \gtrsim \|f \|_* \gtrsim \|f\|_{\B}$ for all $f \in H^{\infty}$; see \cite{Zh}.

Also if $h: \D \rightarrow \mathbb C$ is univalent, then $h \in \text{BMOA}$ if and only if $h \in \B$, and $h \in \text{VMOA}$ if and only if $h \in \B_0$; see, e.g., \cite[p. 283]{Gar}.

Next, we review some fundamental results about the structure of composition semigroups.  A reference for these is Section 3 of \cite{BP};
see also Section 3 of \cite{Sis}. Every nontrivial semigroup of analytic functions $\{\phi_t\}_{t\ge0}$ has a unique common fixed point $b$ with $|\phi_t'(b)| \leq 1$ for all $t \geq 0$, called the Denjoy-Wolff point of $\{\phi_t \}_{t \geq 0}$. Under a normalization, the Denjoy-Wolff point $b$ may be assumed to be 0 or 1.
If $b = 0$, then
\begin{align}\label{b=0}
\phi_t(z) = h^{-1}(e^{-ct} h(z)),
\end{align}
where $h$ is a univalent function from $\D$ onto a spirallike domain $\Omega$, $h(0) = 0$,  $\re c\ge0$,
and $we^{-ct} \in \Omega$ for each $w \in \Omega, t \geq 0$.
If $b = 1$, then
\begin{align}\label{b=1}
\phi_t(z) = h^{-1}( h(z) + ct ),
\end{align}
where $h$ is a univalent function from $\D$ onto a close-to-convex domain $ \Omega$, $h(0) = 0$, where  $\re c\ge0$,
and  $w + ct \in \Omega$ for each $w \in \Omega, t \geq 0$.

If $\{\phi_t\}_{t\ge0}$ is a semigroup, then each map $\phi_t$ is univalent. The infinitesimal generator of $\{\phi_t\}_{t\ge0}$ is the function
$$G(z) = \lim_{t \rightarrow 0^+} \frac{\phi_t(z) - z}{t}, \; z\in \D.$$
This convergence holds uniformly on compact subsets on $\D$ so $G \in H(\D)$. $G$ has a representation
$$G(z) = (\overline{b}z - 1)(z - b) P(z), \; z \in \D,$$
where $b$ is the Denjoy-Wolff point of $\{\phi_t\}_{t\ge0}$, $P \in H(\D)$ with $\re P(z) \geq 0$ for all $z \in \D$.

We will use the following result from  \cite{Bla2}:
\begin{thm}\label{genfcn} \cite[Theorem 1]{Bla2}
Let  $\{\phi_t\}_{t\ge0}$ be a semigroup with generator $G$ and $X$ a Banach space of analytic functions which contains the constant functions and such that $\sup_{0 \leq t \leq 1} \|C_t\| < \infty$.  Then
$$[\phi_t , X] = \overline{ \{f \in X: Gf' \in X \}}.$$
\end{thm}

We will now review some basic facts about prime ends introduced by Carath\'eodory in order to describe the boundary behavior of a univalent function $h$ from $\D$ onto a simply connected domain $\Omega \subset \mathbb{C}  \cup \{ \infty \}$; see Section 2.4 in \cite{Pom}.
A {\it{crosscut}} $C$ of $\Omega$ is an open Jordan arc in $\Omega$ such that $\overline{ C} \backslash C$ consists of one or two points on $\partial \Omega$.  Here $\overline{ C}$ denotes the closure of $C$ in the Riemann sphere.
If $C$ is a crosscut of $\Omega$, then $\Omega \backslash C$ has exactly two components.  The diameter of a set $E\subset \mathbb{C} \cup \{\infty\}$ in the spherical metric is denoted $\text{diam}^{\#}\;E$.


A {\it{null-chain}} $(C_n)_{n \ge 0}$ of $\Omega$ is defined as a sequence of crosscuts of $\Omega$ such that

\begin{enumerate}
\item[(a)] $\overline{C_n} \cap \overline{C_{n+1}} = \emptyset$ for all $n$;
\item[(b)] $C_n$ separates $C_0$ and $C_{n+1}$ for all $n$;
\item[(c)] $\text{diam}^{\#} \; C_n \rightarrow 0$ as $n \to \infty$.
\end{enumerate}

Let $V_n$ be the component of $\Omega\setminus C_n$ not containing $C_0$, and define $V'_n$ similarly for $(C'_n)$.  The null-chains $(C_n)$ and $(C'_n)$ are called {\it{equivalent}} if, for every sufficiently large $m$, there exists $n$ such that $V_n \subset V'_m$ and
$V'_n \subset V_m$.  This is an equivalence relation on the set of all null-chains of $\Omega$.  The equivalence classes are called the {\it{prime ends}} of $\Omega$.  A point $a \in \mathbb C \cup \{\infty\}$ is called a {\it{principal point}} of the prime end $P$ if there exists a null-chain $(C_n)$ representing $P$ such that $C_n \rightarrow \{a\}$ in the spherical metric as $n \rightarrow \infty$.  The set
$I(P) = \bigcap_n  \overline{V_n}$ is non-empty, compact and connected in $\mathbb C \cup \{\infty\}$. We call $I(P)$ the {\it{impression}} of $P$. If $I(P)$ is a single point we call the prime end {\it{degenerate}}.

We call a prime end $P$ {\it{accessible}} if there exists a Jordan arc that lies, except for one endpoint on $\partial \Omega$, in $\Omega$ and intersects all but finitely many crosscuts of every null-chain $(C_n)$ that represents P.

 We will also need the following result from univalent function theory which states that univalent functions have no Koebe arcs. For our purposes, it may be stated as follows:

\begin{lem} \label{koebearcs} \cite[Lemma 9.3 and Corollary 9.1]{Pom2}
Suppose that $h:\D \to \mathbb{C}$ is univalent, $\{\eta_n\}$ is a sequence of Jordan arcs in $\D$, and $h(\eta_n)$ converges to a point $w_0 \in \mathbb{C} \cup \{ \infty \}$, i.e.,
  $$h(z) \to w_0, \quad z \in \eta_n, \quad n \to \infty.$$
Then the Euclidean diameter of $\eta_n$ satisfies $\diam \eta_n \to 0$, as $n\to\infty$.
\end{lem}

\section{Proof of Theorem \ref{bmoathm}} \label{bmoasec}


\begin{thm} \label{bmoa}
Given any nontrivial semigroup $\{\phi_t\}$, there exists $f \in H^{\infty}$ such that
$$ \liminf_{t\to0}\| f \circ \phi_t - f \|_{\B} \geq 1.$$
\end{thm}

\begin{proof}

Let $\{\phi_t\}$ be a nontrivial semigroup, and
let $b$ be the corresponding Denjoy-Wolff point.
After normalization, we may assume that $b$ is either 0 or 1.  First we deal with the case that $b = 0$,
so that each $\phi_t$ is given by (\ref{b=0}).

When $\re c = 0$ in (\ref{b=0}), the $\{ \phi_t \}$ are rotations of the disk.
An infinite interpolating Blaschke product $f$ with all its zeros on the radius $[0,1)$ satisfies
  $$\limsup_{r \to 1^-} |f'(r)|(1-r) \geq \delta$$
for some $\delta > 0$; see \cite[Lemma 2.8]{AJS}.   However, as $f$ extends to be analytic in a neighborhood of every point $e^{i \theta} \in \partial \D$ except 1, we have
  $$\limsup_{r \to 1^-} |f'(re^{i\theta})| ( 1 - r ) = 0 \qquad (e^{i\theta} \neq 1).$$
If $\phi_t(z) = ze^{iat}$ for real $a \neq 0$, then, for all $t$ between 0 and $2\pi/|a|$,
\begin{align*}
 \| f \circ \phi_t - f \|_{\B} &\geq \sup_{0<r<1} \left[|f'( \phi_t( r))\phi'_t(r) - f'(r)|(1-r)\right] \\
  &\geq \delta.
\end{align*}
Replacing $f$ with $f/\delta$ gives the result.

Next consider the case that $\re c > 0$  in (\ref{b=0}), so that $\{ \phi_t \}$ does not consist of automorphisms.  Since $\Omega$ is spirallike about 0, we can choose $w_0 \in \partial \Omega$ such that
  $$|w_0| = \inf \{|w| : w \in \partial \Omega\}.$$
 Then $[0,w_0)\subset\Omega$.  For all sufficiently large values of $n$, let $C_n$ be the connected component of $\{w \in \Omega : |w - w_0| = 1/n\}$ that intersects $[0,w_0)$.  Then $(C_n)$ is a null-chain that represents an accessible prime end $P$ with principal point $w_0$. By \cite[Corollary 9.3]{Pom2} (see also \cite{Pom}), $\lim_{r \to 1^-} h(r\gamma_0)$ exists (and is equal to $w_0$), where $\gamma_0 \in \partial \D$ corresponds to $P$.
Thus,
  $$\lim_{r \to 1^-} \phi_t(r\gamma_0) = h^{-1}( e^{-ct}w_0) \in \D, \quad t >0 .$$
Since $\phi_t$ is univalent and bounded, $\phi_t$ is in the Dirichlet space, and
$\phi_t \in \mathcal B_0.$  Hence
  $$\lim_{r \to 1^-}  |\phi'_t(r\gamma_0)|(1-r) = 0.$$
Letting $f$ be an infinite interpolating Blaschke product with zeros confined to the radius
$\{ r\gamma_0 : 0 < r < 1 \}$,
we have
  $$\limsup_{r \to 1^-} |f'(r\gamma_0)|(1-r) \geq \delta$$
for some $\delta > 0$.  However, $f'$ is continuous on $\D$, so for fixed $t>0$
  $$\lim_{r \to 1^-} |f'(\phi_t(r\gamma_0))| =  |f'(h^{-1}(e^{-ct}w_0))| < \infty.$$
Thus, for all $t > 0$,
\begin{align}\label{deltabmoa}
\|f\circ\phi_t - f\|_{\B} &\geq \limsup_{r \to 1^-} |f'(\phi_t(r\gamma_0))\phi'_t(r\gamma_0) - f'(r\gamma_0)|(1-r)
 \geq \delta,
\end{align}
and replacing $f$ by $f/\delta$ gives $ \| f \circ \phi_t - f \|_{\B} \geq 1$.

It remains to consider the case that the Denjoy-Wolff point $b=1$, so that the $\phi_t$ are
given by $\phi_t(z) = h^{-1}(h(z)+ct)$ from (\ref{b=1}).
If the $\phi_t$ are automorphisms, then the map $w \mapsto w + ct$ is an automorphism of $\Omega$.  It follows that $\Omega$ is a half-plane or strip, and $\partial \Omega$ in $\Bbb{C}$ consists of impressions of degenerate prime ends which are not fixed under $w \mapsto w+ct$, $t > 0$. Let
$w_0 \in \mathbb{C}$ be one such impression, and let $\gamma_0$ be the corresponding point in $\partial \D$.
Then  $\phi_t(\gamma_0) \in \partial \D$ but $\phi_t(\gamma_0) \neq \gamma_0$ for all $t > 0$.  Let
$f$ be a bounded function (such as an interpolating Blaschke product with zero sequence along
$\{r \gamma_0 : 0 < r < 1\}$) such that
  $$\limsup_{r \to 1^-} |f'(r \gamma_0 )|(1-r) \geq \delta,$$
for some $\delta > 0$, but
  $$\lim_{r \to 1^-} |f'(r \gamma )|(1-r) = 0$$
for all $ \gamma \in \partial \D, \gamma \neq \gamma_0$.  We require that $f$ is well-behaved away from $\gamma_0$, i.e., $f'$ extends continuously to $\overline{\D} - \{\gamma_0\}.$
Now fix some $t>0$.
Since $\gamma_0$ is not a fixed point of $\phi_t$, composition with $\phi_t$ moves the radius $[0,\gamma_0)$ away to where $f$ is well-behaved.  For $\gamma_t = \phi_t(\gamma_0)$, our requirement ensures
that $f'$ extends to be continuous at $\gamma_t$, so
  $$\lim_{r \to 1^-} f'(\phi_t(r \gamma_0)) = f'( \gamma_t).$$
Since $\phi_t$ is an automorphism, $\phi_t'$ is bounded on $\D$.  For fixed $t > 0,$
\begin{align*}
\|f\circ\phi_t - f\|_{\B} &\geq \limsup_{r \to 1^-} |f'(\phi_t(r\gamma_0))\phi'_t(r\gamma_0) - f'(r\gamma_0)|(1-r)\\
&\geq \delta.
\end{align*}
As before, replacing $f$ by $f/\delta$ gives the result.

In the non-automorphism case, for $t>0$ the map given for $w \in \Omega$  by $w \mapsto w + ct$ is not onto.
Let $t>0$ and $w\in\Omega\setminus \left(\Omega + ct\right)$. Then there is $t_0\in(0,t]$
such that $w_0=w-ct_0\in\partial\Omega$, but $(w_0,w]\subset\Omega$.  As in the case $b=0$,   $w_0$
is the principal point of an accessible prime end,
and the same argument terminating with (\ref{deltabmoa}) completes the proof.

\end{proof}

\noindent{\bf \em Proof of Theorem \ref{bmoathm}:}
Each test function $f$ in Theorem \ref{bmoa} is in $H^{\infty}$, and hence in $X$ from the
hypothesis that $H^\infty\subseteq X$.  Since $X\subseteq \B$,  the
Closed Graph Theorem shows that $\|\cdot\|_\B\lesssim \|\cdot\|_X$
and bounding the Bloch norm away from 0 bounds the $X$ norm as well.
Thus it follows from Theorem \ref{bmoa}  that $f\notin [\phi_t, X]$, and so $[\phi_t, X ]\subsetneqq X$.
\hfill $\square$

\section{Proof of Theorem \ref{uniformthm} and Corollaries} \label{Asec}

\noindent{\bf \em Proof of Theorem \ref{uniformthm}:}
In the case where the Denjoy-Wolff  point $b$ of $\{\phi_t\}$ is on the boundary of $\D$, we may assume $b = 1$ and
  $$\phi_t(z) = h^{-1}(h(z) + ct), \quad z \in \D.$$
As before, $h$ is a univalent function from $\D$ to a close-to-convex domain $\Omega$ which has the property
 that $w+ct\in\Omega$ for each $w\in\Omega$ and all $t>0$, for some
 $c \in \mathbb{C}$ with $\re c \geq 0$.  If $c=0$, the result is trivial.  So assume  $c \neq 0.$

Suppose $\phi_t(z)$ does not converge uniformly to $z$ in $\D$.  Then there exist some $\delta > 0$ and infinite sequences $\{t_n\}, t_n \to 0^+$ as $n \to \infty$, and $\{z_n\} \subset \D$ such that
  $$\delta \leq |\phi_{t_n}(z_n) - z_n|, \quad n\ge 1.$$
  Letting $w_n = h(z_n) \in \Omega$,
\begin{align*}
|\phi_{t_n}(z_n) - z_n| &= |h^{-1}(h(z_n) + ct_n) - h^{-1}(h(z_n))|\\
&= |h^{-1}(w_n + ct_n) - h^{-1}(w_n)|.
\end{align*}
The points $w_n + ct_n$ and $w_n$ are endpoints of a line segment in $\Omega$ which pulls back to the Jordan arc
  $$\eta_n = \{h^{-1}(w_n+ct): 0 \leq t \leq t_n\} \subset \D.$$
Since $t_n \to 0$ and $\overline{\Omega}$ is compact in the Riemann sphere, we may pass to a subsequence of $\{w_n\}$ and assume the line segment $[w_n,w_n+t_n] = h(\eta_n) \to w_0 \in \overline{\Omega} \cup \{\infty\}$.  However, $\text{diam } \eta_n \geq \delta$, contradicting Lemma $\ref{koebearcs}$.  Therefore, $|\phi_t(z) - z| \to 0$ uniformly in $\D$ as $t \to 0^+$.

In the case where $\{\phi_t\}$ has Denjoy-Wolff point inside $\D$, minor modifications of the same argument yield uniform convergence.
We may assume
  $$\phi_t(z) = h^{-1}(h(z)e^{-ct}), \quad z \in \D, t \geq 0,$$
where $\re c\ge0$ and $\Omega = h(\D)$ is
 a spirallike domain. Supposing that uniform convergence fails, we have
  $$\delta \leq |\phi_{t_n}(z_n) - z_n| = |h^{-1}(w_ne^{-ct_n}) - h^{-1}(w_n)|.$$
Let $\eta_n = \{h^{-1}(w_ne^{-ct}): 0 \leq t \leq t_n\}$.
Passing to a subsequence as before, $h(\eta_n) \to w_0 \in \overline{\Omega} \cup \{\infty\}$
while $\text{diam } \eta_n \geq \delta$, a contradiction of Lemma \ref{koebearcs}. Thus, $\phi_t(z) \to z$ uniformly in $\D$.
\hfill $\square$

\vskip .2in

We can now prove Corollaries \ref{polythm}, \ref{hinftythm}, and  \ref{Athm}.
\vskip .1in
\noindent{\bf \em Proof of Corollary \ref{polythm}:}
The closed graph theorem and Theorem \ref{uniformthm} show that
$$\lim_{t\to0}\|\phi_t - z \|_X \lesssim \lim_{t\to0}\| \phi_t - z \|_{H^{\infty}} = 0,$$
i.e., the function $f(z) = z$ is in $[\phi_t, X]$.  To extend this to any function in $X_{\mathcal P}$ is straightforward.
The inequality
  $$\|\phi_t(z)^n - z^n\|_{\infty} \leq n \|\phi_t(z) - z\|_{\infty}$$
and linearity show any polynomial is in $[\phi_t, X]$, and since $[\phi_t, X]$ is closed,
$X_{\mathcal P} \subset [\phi_t, X]$. \hfill $\square$

\medskip

Corollary \ref{hinftythm} follows immediately from Corollary \ref{polythm} since $A$ is the closure of the polynomials in $H^{\infty}$.

\medskip

\noindent{\bf\em Proof of Corollary \ref{Athm}:}
If $A= [\phi_t, A]$, then the semigroup $\{\phi_t\}$ acts on $A$ and hence
$\{\phi_t\}\subset A$.  Next, for the converse, assume that $\{\phi_t\}\subset A$.
Since the norm on $A$ is the same as the $H^\infty$ norm, Theorem \ref{uniformthm} shows that
the function $f(z) = z$ is in $[\phi_t, A]$.  As in the previous proof, this extends to any function in $A$,
and so $A\subset [\phi_t, A]$. Since the reverse inclusion is trivial,
$A= [\phi_t, A]$ and
the proof of the corollary is complete.
\hfill $\square$

\medskip

In
Proposition \ref{Anotbracketsp} below, we will  show that  there are semigroups $\{\phi_t\}$
such that $A \subsetneqq [\phi_t,H^{\infty}]$.  In the other direction, it is known that
when $\{\phi_t\}$ consists of rotations or dilations of $\D$, then VMOA$\,=[\phi_t,\text{BMOA}]$;
see \cite{Bla}.  In the setting of $H^\infty$ we have the following analog:

\begin{prop}
If $\{\phi_t\}$ consists of rotations or dilations of $\D$, then $A = [\phi_t,H^{\infty}]$.
\end{prop}

\begin{proof}
In light of Corollary \ref{hinftythm} we need only show $[\phi_t,H^{\infty}] \subset A$ for such $\{\phi_t\}$.
If the semigroup consists of rotations, then $\phi_t(z) = e^{iat}z$ for some real $a \neq 0$.  If $f \in [\phi_t,H^{\infty}]$, then
   \begin{equation} \label{fdelta}
   \lim_{t \to 0^+} \|f(e^{iat} z) - f(z)\|_{\infty} = 0.
   \end{equation}
It follows from (\ref{fdelta}) that $f \in A$; see the proof of
Theorem 3.5 in \cite{AJS}.  Hence $[ \phi_t \, H^{\infty} ] \subseteq A$ as required.

If $\{\phi_t\}$ consists of dilations, they have the form $\phi_t(z) = e^{-ct}z$ for $\re c > 0$, and any $f \in H^{\infty}$ is uniformly continuous on $\phi_t(\D)$ for $t > 0$. I.e., $f \circ \phi_{t} \in A$.  Thus, $f \in [\phi_t,H^{\infty}]$
means $f$ is the uniform limit of the functions $f\circ\phi_t$, implying $f \in A$.
\end{proof}

Next, we turn to the demonstration that there is a semigroup $\{\phi_t\}$ with $A\subsetneqq[\phi_t,H^\infty]$.
First we need a proposition.

\begin{prop} \label{BprimeGbdd}
There exists an infinite Blaschke product $B$ such that
$$B'(z)(1-z) \in H^{\infty}.$$
\end{prop}

\begin{proof}
Let $B$ be the interpolating Blaschke product whose zeroes are $a_k = 1 - 2^{-k}$ for $k = 1,2,...$.
We write $B = \prod_1^{\infty} \sigma_k$, where
  $$\sigma_k(z) = \frac{a_k - z}{1 - a_k z} \qquad (z \in \overline{\D}).$$
Let
  $$f(\theta) = \sum_{k=1}^{\infty} \frac{1-a_k}{(1-a_k)^2 + \theta^2} \qquad (-\pi \leq \theta < \pi).$$
As in the proof of \cite[Lemma 1]{AhC} , $|e^{i\theta} - a_k|^2 = (1-a_k)^2 + 4a_k\sin^2(\theta/2).$
Hence, for $0 < |\theta| < \pi/4$,

\begin{align*}
|B'(e^{i\theta})| &\leq \sum_{k = 1}^{\infty} \left( |\sigma'_k(e^{i\theta})| \prod_{j \neq k} |\sigma_j(e^{i\theta})| \right)\\
&\leq \sum_{k=1}^{\infty} \frac{1-a_k^2}{|1-a_ke^{i\theta}|^2}
 \lesssim f(\theta).
\end{align*}

With $a_k = 1-2^{-k}$ and $\theta \neq 0$, we have

\begin{align}
f(\theta) &= \sum_k \frac{2^{-k}}{2^{-2k} + \theta^2}\\
&= \sum_{2^{-2k} \geq \theta^2} \frac{2^{-k}}{2^{-2k} + \theta^2} + \sum_{2^{-2k} < \theta^2} \frac{2^{-k}}{2^{-2k} + \theta^2} \label{sums1}\\
&\approx \sum_{2^{-k} \geq \theta} \frac{2^{-k}}{2^{-2k}} + \sum_{2^{-k} < \theta}\frac{2^{-k}}{\theta^2} \label{sums2}\\
&\leq \frac{2}{\theta} + \frac{2\theta}{\theta^2} = \frac{4}{\theta}.
\end{align}

The approximation from (\ref{sums1}) to (\ref{sums2}) invokes the inequality
  $$\frac{1}{x} \geq \frac{1}{x+y} \geq \frac{1}{2x},  \quad 0 < y \leq x.$$
The sums in (\ref{sums2}) are geometric, which we estimate by their largest terms.  Hence
 $$f(\theta) \lesssim \frac{1}{\theta}, \quad 0 < |\theta| \leq \pi/4.$$

Note that $B'$ extends to be continuous on $\overline{\D \setminus D_{\eps}}$, where $D_{\eps}$ is a small disk centered at 1. Also $|1-e^{i\theta}| \approx |\theta|$,  for $0 < |\theta| \leq \pi$. Hence there is $M > 0$
such that
  $$|(1-e^{i\theta})B'(e^{i\theta})| \lesssim |\theta| f(\theta) \leq M, \quad 0 < |\theta| \leq \pi.$$
Applying the same estimates to partial sums we see the finite Blaschke products $B_n = \prod_1^n\sigma_k$ satisfy
$|B'_n(e^{i\theta})(1-e^{i\theta})| \leq M$ for all $\theta$, and by the maximum principle
$|B'_n(z)(1-z)| \leq M$ for all $z \in \D$.  Also, $B_n'(z) \to B'(z)$ pointwise on $\D$.
Thus for $z \in \D$,
  $$|(1-z)B'(z)| = \left| \lim_{n\to \infty}(1-z)B_n'(z)\right|  \leq M.$$

Thus, $B'(z)(1-z) \in H^{\infty}$.
\end{proof}

\begin{prop} \label{Anotbracketsp}
There exist semigroups $\{\phi_t\}$ such that $A \neq [\phi_t,H^{\infty}].$
\end{prop}

\begin{proof}
The function $G_1(z) = (1-z)$ is the (infinitesimal) generator of a semigroup with $\phi_t(z) = e^{-t}z + 1-e^{-t}$ and Denjoy-Wolff point on $\partial \D$.  Also, $G_2(z) =  -z(1-z)$ is the generator of a semigroup with $\phi_t(z) = \frac{e^{-t}z}{(e^{-t}-1)z + 1}$ and inner Denjoy-Wolff point; see \cite{Sis}.
  Let $B$ be the Blaschke product from
 Proposition \ref{BprimeGbdd}.
 Then  Proposition \ref{BprimeGbdd} shows that $G_1B'$ and $G_2B'$ are in $H^\infty$.
Whether $\{\phi_t\}$ is generated by $G_1$ or $G_2$,  Theorem \ref{genfcn}
implies $B \in [\phi_t,H^{\infty}]$
 for the Blaschke product $B \notin A$.
\end{proof}

{\em Acknowledgment.} The authors want to thank the referee for carefully reading the paper and for the valuable
suggestions that improved it.


 \end{document}